\documentclass[reqno,11pt,centertags]{article}
\usepackage{amsmath,amsthm,amscd,amssymb,latexsym,upref}
\date{\today}
\input epsf
\usepackage{epsfig}
\usepackage[T2A,OT1]{fontenc}
\usepackage[ot2enc]{inputenc}
\usepackage[russian,english]{babel}
\usepackage[margin=3.5 cm]{geometry}
\usepackage{epic,eepicemu}
\usepackage[all]{xy}

\def\sb{\subset}      \def\sbe{\subseteq}

\def\z{\zeta}

\def\r{\rho}

\def\s{\sigma}

\def\t{\tau}

\def\o{\omega}

\def\O{\Omega}
\def\O{\Omega}

\def\pa{\partial}

\def\sm{\setminus}

\newcommand{\NN}{{\mathbb{N}}}
\newcommand{\TT}{{\mathbb{T}}}
\newcommand{\DD}{{\mathbb{D}}}
\newcommand{\EE}{{\mathbb{E}}}
\newcommand{\RR}{{\mathbb{R}}}
\newcommand{\CC}{{\mathbb{C}}}
\newcommand{\PP}{{\mathbb{P}}}

\def\la{\langle}
\def\ra{\rangle}

\def\bydef{\,\lower-.1ex\hbox{\rm :}\!=}

\def\cF{{\cal F}}

\def\cH{{\cal H}}

\def\cS{{\cal S}}

\def\Id{\mbox{ \rm Id} }

\allowdisplaybreaks \numberwithin{equation}{section}
\newtheorem{theorem}{Theorem}[section]

\newtheorem{lemma}[theorem]{Lemma}

\theoremstyle{definition}

\newtheorem{remark}[theorem]{Remark}

\date{\today}

\begin{document}

\author{Paul F.X. M\"uller  and Peter Yuditskii}

\title
{Interpolation for Hardy Spaces: Marcinkiewicz decomposition, Complex Interpolation and Holomorphic Martingales}

%\thanks{Insert here the institutions and organizations that supported the research activity corresponding to the paper.}

\maketitle

% If it is necessary include here a short version for the title of the paper and for the list of authors.
% In that case remove also the % symbol comment mark of the following line.
\markboth{\em P.F.X. M\"uller and P. Yuditskii}{\em Interpolation for Hardy Spaces}

\abstract{The  real and complex interpolation spaces for the classical Hardy spaces $H^1$ and $H^\infty$ were determined in 1983 by P.W. Jones. 
Due to the analytic constraints the associated Marcinkiewicz decomposition gives  rise to a delicate approximation problem for the $L^ 1$ metric.
Specifically for $ f \in H^p$ the size  of 
$$ {\rm{inf}} \{ \| f - f_1 \| _1 \,:\, f_1 \in H^\infty ,\, \|f_1\|_\infty \le \lambda \}$$ 
needs to be determined  for any $    \lambda \in \RR_+. $ In the present paper we develop a new set of truncation formulae for obtaining 
the  Marcinkiewicz decomposition of  $(H^1, H^\infty) $. We revisit the real and complex interpolation theory for Hardy spaces by examining our newly found formulae.}

\bigskip
\textbf{Keywords:} {Hardy spaces, Holomorphic martingales, complex and real interpolation spaces, Marcinkiewicz decomposition}

\bigskip

%  Put here your AMS Classifications instead of the ones of the example with
%  the format \classification{primary}{secondary}

\textbf{MSC:} {60G42, 60G46} {32A35}

%%%%%%%%%%%%%%%%%%%%%%%%%%%%%%%%%%%%%%%%%%%%%%%%%%%%%%%%%%%%%%%%%%%%%%%%%%%%%%%%%%%%%%
%%%%%%%%%%%%%%%%%%%%%%%%%%%%%%%%%%%%%%%%%%%%%%%%%%%%%%%%%%%%%%%%%%%%%%%%%%%%%%%%%%%%%%
%       Body of the paper
%%%%%%%%%%%%%%%%%%%%%%%%%%%%%%%%%%%%%%%%%%%%%%%%%%%%%%%%%%%%%%%%%%%%%%%%%%%%%%%%%%%%%%
%%%%%%%%%%%%%%%%%%%%%%%%%%%%%%%%%%%%%%%%%%%%%%%%%%%%%%%%%%%%%%%%%%%%%%%%%%%%%%%%%%%%%%

\section{Introduction}
Let $ \DD \sb \CC $ denote the open unit disk, and let $ \TT $ denote its boundary.  
As usual, $H^p(\TT)$ denote the Hardy spaces in the disk.
We let $ ( \O , ( \cF_t ) ,\PP ) $
denote the filtered Wiener space;
we will work   with the Hardy spaces of holomorphic random variables on Wiener space $H ^ p( \O )$, see \S 3 for details.
In this paper we revisit   the real and complex interpolation method for the couples  $(H^1(\TT) , H^\infty( \TT)) $ and  $(H^1(\O), H^\infty(\O)) $.
We introduce new truncation formulae to produce a Marcinkiewicz decomposition for $(H^1(\TT) , H^\infty( \TT)) $ and  exploit those  in combination with  tools of 
Stochastic Analysis, such as holomorphic random variables and stopping time decompositions. 
In order to make this paper accessible for non-specialists   we included an extended  discussion of the concepts and tools we employed here.  
   
Let $(V, \|\cdot\|_V)$ be a Banach space and let $X_0$ and $X_1$ be linear subspaces of $V$. Assume  that $(X_0, \|\cdot\|_0)$ and $(X_1,\|\cdot\|_1)$ are Banach spaces and that the formal inclusion maps $X_i\to V$ define bounded linear operators. We then say that $(X_0,X_1)$ form a compatible pair of Banach spaces. We recall now the construction of the complex and real interpolation spaces associated to a compatible pair $(X_0,X_1)$.

We first define the complex interpolation spaces associated to a compatible pair  $(X_0,X_1)$. Let
$S=\{\zeta\in \CC:\ 0<\Re \zeta<1\}$, let $\overline S:=S\cup S_0\cup S_1$ denote its closure, where $S_0 =\{it:t\in \RR\}$, and 
$S_1 =\{1 + it:\  t \in \RR\}$. Define $\cF(X_0, X_1)$ to be the vector space of all functions
 $F : \overline{S} \to X_0 + X_1$, satisfying
 
 \begin{enumerate}
 \item $F$ is bounded and continuous on $\overline{S}$,
 \item $F$ is analytic  on ${S}$,
 \item $F(S_0)\subseteq X_0, F(S_1)\subseteq X_1$ and the respective restrictions $F : S_0\to X_0$, and
  $F : S_1 \to X_1$ are continuous.
 \end{enumerate}
Equipped with the norm $\|F\|_{\cF} =\max\{\sup_{t\in\RR}\|F(it)\|_{X_0}, \sup_{t\in\RR}\|F(1+it)\|_{X_1}\}$
the space $\cF(X_0,X_1)$ is a Banach space. Let $0 < \theta < 1$. The complex interpolation is defined as
$$
(X_0,X_1)_{[\theta]} =\{x\in X_0 +X_1:\ \exists F \in\cF(X_0,X_1), F(\theta)=x\},
$$
equipped with the natural quotient norm $\|x\|_{[\theta]} = \inf\{\|F\|_{\cF(X_0,X_1)}:\ F(\theta) = x\}$. 

For the compatible couple of Banach spaces $(L^1,L^\infty)$ the complex interpolation spaces are determined by the M. Riesz Theorem, asserting that
$$
(L^1,L^\infty)_{[\theta]} =L^p,\ 1/p=1-\theta
$$
with equality of norms, $\|x\|_{[\theta]}=\|x\|_{L^p}$, for any $x \in L^1 + L^\infty$.

Now we turn to defining the family of real interpolation spaces of a compatible pair
$(X_0 , X_1 )$ of Banach spaces. Recall that for $x \in X_0 + X_1$ and $t > 0$, the $K$-functional with respect to $(X_0, X_1)$ is defined as
$$
K(x,t,X_0,X_1)=\inf\{\|x_0\|_{X_0} + t\|x_1\|_{X_1}:\ x_0 \in X_0,\ x_1\in X_1,\ x=x_0 +x_1\}.
$$
Given $0<\theta<1$ and $1\le q<\infty$ we define $(X_0,X_1)_{\theta,q}$ to consist of those 
$x\in X_0+X_1$ for which
$$
\|x\|_{\theta,q} = \left(\int_0^\infty[t^\theta K(f,t,X_0,X_1)]^q\frac{dt}{t}\right)^{1/q}
$$
is finite. The space $((X_0, X_1)_{\theta,q},\|\cdot\|_{\theta,q})$ is a Banach space. 

The real interpolation spaces of the couple $(L^1, L^\infty)$ coincide with the Lorentz-spaces. There exist $c > 0, C < \infty$ such that for any 
$x\in L^1 + L^\infty$,
$$
c\|x\|_{L^{p,q}}\le \|x\|_{\theta,q}\le C\|x\|_{L^{p,q}},
$$
whenever $1/p = 1-\theta$, and $1 < q < \infty$. Hence we have equality of spaces 
$(L^1, L^\infty)_{\theta,q}= L^{p,q}$, where $1/p = 1-\theta$, and $1 < q < \infty$ with equivalence of norms. By an argument of J. Marcinkiewicz, the identification $(L^1, L^\infty)_{\theta,q} = L^{p,q}$ can be obtained from the following decomposition of $L^p$, where $1<p<\infty$. For $x\in L^p$ with $\|x\|_p\le 1$,
and $\lambda>0$ there exist $x_0\in L^1$ and $x_1 \in L^\infty$ such that
\begin{equation*}\label{}
x=x_0+x_1,\quad \|x_1\|_\infty\le\lambda, \ \text{and}\ \|x_0\|_1\le C_p\lambda^{1-p}.
\end{equation*}
%%%%%%%%%
Putting 
$E=\{t\in\TT: |x|\ge \lambda\}$ and set
$$
x_1(t)=\begin{cases} x(t),& t\not\in  E\\
\lambda,&t\in E
\end{cases}
$$
gives $\|x_1\|_\infty\le \lambda$.
Since with H\"older's inequality 
$
\int|x-x_1|\le 2\int_E|x|\le 2 |E|^{1-1/p},
$
Chebyshev's inequality yields the desired $L^1$-approximation
%%%%%%%%%%
\begin{equation*}%\label{eq11}
\int | x-x_1|\le  {C_p}{\lambda^{1-p}}.
\end{equation*}

We next turn to discussing  Hardy spaces $ H^p(\TT) \sbe L^p(\TT)$. Recall that,
if $ f \in L^p(\TT ) $ then $ f \in H^p(\TT ) $ if and only if, the harmonic extension of $ f $ to $\DD$ gives rise to an  analytic function in $ \DD $.
  P. W. Jones  \cite{MR697611} determined the real and complex interpolation spaces for the 
compatible couple of Banach spaces $(H^1(\TT ), H^\infty(\TT ) ) $ as follows, 
  \begin{equation}\label{15-9-2}
   ( H^1 (\TT ), H^\infty(\TT )) _{\theta, q } = H^{p,q}(\TT ),  \quad      ( H ^ 1(\TT ) , H^\infty(\TT ))  _{[\theta]} = H^p(\TT ),
\end{equation}
where $1/p = 1 -\theta, 1 < p < \infty . $ 
We refer also to  Jones'  \cite{MR760480,MR804737} for a survey of those  results,  and for  extensions thereof.    

To identify the real interpolation spaces for the couple $( H^1 (\TT ), H^\infty(\TT )) $,  P. W. Jones  \cite{MR697611}  established  the 
 Marcinkiewicz decomposition for   $H^ p( \TT )$, where $1< p < \infty :$
For any   $ f \in  H^ p( \TT ) $ with  $\|f\|_p \le 1, $  and $ \lambda > 0 $ there exist $ f_0 \in H^ 1 $ and $f_1\in H^ \infty$ such that 
\begin{equation}\label{14-9-8}  f = f_0 + f_1 , \quad \|f_1\|_\infty \le   \lambda, \quad \text{and}\quad  \|f_0\|_1 \le C_p \lambda^{1-p} . \end{equation}     
(Note that  the Marcinkiewicz decomposition  for the $L^ p $ spaces described above, would not preserve analyticity.)   We refer to the monograph by Bennett and Sharply \cite{MR928802} 
for an exposition of Jones' approach to the Marcinkiewicz decomposition for Hardy spaces.

We next discuss the complex interpolation spaces of $( H^1, H^\infty) $.  In  \cite{MR697611} Jones proved that 
for any   $ f \in  H^ p( \TT ) $ there exists $F_1 \in \cF ( H^1 , H^\infty)$  such that 
$$
\| F_1 \|_{\cF ( H^1 , H^\infty)} \le C \|f\|_p 
$$ 
and $\| F_1(\theta) - f \|_p \le (1/2)\|f\|_p $ where $ 1/p = 1-\theta.$  Replacing $ f $ by 
 $F_2(\theta) - f$  and iterating gives a sequence $F_n \in   \cF ( H^1 , H^\infty)$ satisfying, 
$$ \| F_n \|_{\cF ( H^1 , H^\infty)} \le C 2^{-n}\|f\|_p \quad\text{and}\quad   \| F_{n+1}(\theta) -  F_{n}(\theta) \|_p \le (1/2)^{n+1}\|f\|_p .$$
Jones then puts  $ G = \lim_{m \to \infty} G_m $ where   $ G_m = \sum_{n=1}^m F_n $, 
 and thereby obtains  $G \in \cF ( H^1 , H^\infty)$  satisfying
$ \| G \|_{\cF ( H^1 , H^\infty)} \le 2C \|f\|_p $ and the pointwise constraint   $ G(\theta)=f  $. 
In view of the M. Riesz theorem this yields  $ ( H ^ 1(\TT ) , H^\infty(\TT ))_{[\theta]} = H^p(\TT )$, where $ 1/p = 1 -\theta . $

The course of developement  gave rise to several proofs of Jones' interpolation theorems, for instance  in the work by  G. Pisier  \cite{MR1178030,MR1232844},  Q. Xu  \cite{MR1196097},
S.  Kislyakov \cite{MR1707360, MR1152601}, S. Kislyakov and  Q. Xu 
\cite{MR1758858}, P.W. Jones and P.F.X. M\"uller \cite{jopfxm},  and in \cite{pfxm1, MR1264825}.

In this  paper we add  a new angle  to  the Marcinkiewicz decomposition for   $   H^ p( \TT ) $.   
We exploit the inner-outer factorization in the space $H ^ p ( \TT )$ 
and  reduce the approximation problem \eqref{14-9-8} to the special case where $ f \in H ^ p( \TT ) $ is an outer function. 
We  write down  a  truncation formula that is specifically adjusted to the case of outer functions. The resulting integral estimates are reduced to 
Lemma   \ref{l12}.  In the first part of the paper we used Kislyakov's approach as our point of reference.

In the second part of the paper we illustrate  further the use   of  our newly found  formulae, 
and revisit the martingale approach to identifying the complex interpolation spaces for the Banach couple $ (H ^ 1( \TT ), H ^ \infty ( \TT ) )$.
We will work   with the Hardy spaces of holomorphic random variables on Wiener space $H ^ p( \O )$. Using the truncation formulae 
introduced in the first part of the paper, we revisit the  stopping time decompositions in  \cite{MR1264825}   to prove   that 
 $ ( H ^ 1(\O ) , H^\infty(\O ))_{[\theta]} = H^p(\O )$ where $ 1/p = 1 -\theta . $  Doob's embedding  $ M : H ^ p( \TT ) \to H ^ p( \O ) $ 
and projection  $ N : H ^ p( \O ) \to H ^ p( \TT) $     are the  operators 
by which N. Th. Varopoulos  \cite{var1} relates   Hardy  spaces of holomorphic random variables  to  Hardy spaces of analytic functions:
We have $ N M f = f$ for any $ f \in  H ^ p( \TT )$ and 
$   \| M \|_p   \| N \|_p  = 1 $ for $ 1 \le p \le \infty. $ By means of the operators $ M, N $ the identity 
 $ ( H ^ 1(\O ) , H^\infty(\O ))_{[\theta]} = H^p(\O )$ yields  the complex interpolation spaces in  \eqref{15-9-2}.

\section{Real Interpolation Spaces $(H^1(\TT),H^\infty(\TT))_{\theta}$}

\begin{theorem}\label{th11}
 There exists $C_p$, such that for an arbitrary $\lambda>0$ and
 $f\in H^p(\TT)$,  $\|f\|_{p}=1$, one can find a function $f_1\in H^\infty(\TT)$ such that \eqref{14-9-8} holds.
\end{theorem}

First, we recall some basic facts on Hardy spaces, see e.g. \cite[Chap. I--II]{GA}. 
A function $f(\zeta)$ meromorphic in $\DD$ is said to be of bounded characteristic if
$$
\sup_{0<r<1}\left\{\int_{\TT} \log^+|f(rt)|dm(t)\right\}<\infty,
$$
where $dm$ is the Lebesgue measure on $\TT$. It can be represented as a ratio of two  holomorphic functions
bounded in the disk, that is,
$$
f(\zeta)=\frac{f_+(\zeta)}{f_-(\zeta)},\quad \sup_{\zeta\in\DD}|f_\pm(\zeta)|\le 1.
$$
Such functions form the so-called Schur  class $\cS$ (the unit ball of $H^\infty(\TT)$), in short $f_{\pm}\in\cS$. Functions  from $\cS$ are represented in the form 
$$
f_{\pm}(\zeta)=\prod\frac{\overline{\zeta^\pm_j}}{|\zeta^\pm_j|}\frac{\zeta^\pm_j-\zeta}{1-\zeta\overline{\zeta^\pm_j}}
\exp\left\{ic^\pm+\int_{\TT}\frac{\zeta+t}{\zeta-t}d\tau^{\pm}(t)\right\},
$$
where $\zeta_j^{\pm}\in\DD$, $c^\pm\in\RR$ and $d\tau^\pm(t)$ are positive measures on $\TT$. One can decompose $d\tau^\pm(t)$ into the absolutely continuous $d\tau_{a.c.}^\pm(t)$ and singular $d\tau_{s.}^\pm(t)$ part, $d\tau^\pm(t)=d\tau_{a.c.}^\pm(t)+d\tau_{s.}^\pm(t)$. The factor
$$
f^{out}_{\pm}(\zeta)=\exp\left\{ic^\pm+\int_{\TT}\frac{\zeta+t}{\zeta-t}d\tau_{a.c.}^{\pm}(t)\right\}
$$
is called the \textit{outer} part of the function $f_\pm(\z)$. It is defined uniquely (up to a unimodular constant) via boundary values of the modulus of the given function, $d\tau^\pm_{a.c.}(t)=-\log|f_{\pm}(t)|dm(t)$. The remaining part of the function is called the \textit{inner} part $f^{in}_\pm(\zeta)$, 
so that $f_\pm(\zeta)=f^{in}_\pm(\zeta)f^{out}_\pm(\zeta)$. The inner part contains the Blaschke product and the singular component. The function $f(\zeta)$ is of Smirnov class (or Nevanlinna class $N^+$) if the denominator $f_-$ is an outer function.
Note that any function from $H^p(\TT)$ is a function of Smirnov class, and that functions of Smirnov class obey the maximum principle in the following form: if $f$ is of Smirnov class and its boundary values belong to $ L^p(\TT)$, then 
$f\in H^p(\TT)$.

%%%%%%%%%%

Our proof  of Theorem \ref{th11} is based on the following lemma, which, we believe, is of an independent interest.

\begin{lemma}\label{l12} Let $s$ be of the Schur class and $s(0)>0$. Then 
\begin{equation}
\int |1-s|^q dm\le C_q(1-s(0)), \quad q>1.
\end{equation}
\end{lemma}

\begin{proof}
 Since
\begin{equation}\label{eq21}
\int |1-s|^2dm=\int (1-s-\bar s+|s|^2)dm=2(1-s(0))-\int(1-|s|^2)dm
\end{equation}
our statement is trivial  for $q\ge 2$. So, let $q=2-\delta$, $0<\delta<1$.

Since $\Re (1-s)\ge 0$ the values of $(1-s)^{\delta}$ belong to the angle
$$
\sin\frac{\pi(1-\delta)}2|1-s|^{\delta}\le \Re(1-s)^\delta.
$$
Therefore
\begin{equation*}\label{eq100}
\begin{split}
\sin\frac{\pi(1-\delta)}2&\int \frac{|1-s|^2}{|1-s|^\delta}dm\le
\Re \int \frac{|1-s|^2}{(1-s)^\delta}dm=
\Re \int \overline{(1-s)}(1-s)^{1-\delta}dm\\
=&(1-s(0))+\Re \int \overline{(1-s)}((1-s)^{1-\delta}-1)dm\\
=&(1-s(0))+\Re \int \left(s-1+(1-|s|^2)\right)\frac{(1-s)^{1-\delta}-1} s dm.
\end{split}
\end{equation*}
Since the function $F(\zeta):=\frac{(1-s(\zeta))^{1-\delta}-1} {s(\zeta)}$ is bounded in the unit disk $\DD$, we get
\begin{equation*}\label{eq200}
\begin{split}
\sin\frac{\pi(1-\delta)}2&\int \frac{|1-s|^2}{|1-s|^\delta} dm\le
(1-s(0))\left(1-\frac{(1-s(0))^{1-\delta}-1} {s(0)}\right)\\
+&\Re \int \left(1-|s|^2\right)\frac{(1-s)^{1-\delta}-1} s dm.
\end{split}
\end{equation*}
It follows from the integral representation
$$
\frac{(1-u)^{1-\delta}-1} {u}=\frac{\sin\pi(1-\delta)}\pi\int_1^\infty\frac {(x-1)^{1-\delta}} x \frac{dx}{u-x}
$$
that $\Re F(\zeta)\le 0$. Therefore
\begin{equation*}\label{eq300}
\begin{split}
\sin\frac{\pi(1-\delta)}2\int \frac{|1-s|^2}{|1-s|^\delta}dm\le&
(1-s(0))\left(1+\frac{1-(1-s(0))^{1-\delta}} {s(0)}\right)\\
\le&2(1-s(0)).
\end{split}
\end{equation*}
That is, 
\begin{equation}\label{eq22}
\int |1-s|^q dm\le\frac 2 {\sin\frac{\pi(q-1)}2}
(1-s(0))
\end{equation}
for $1<q\le 2$.
\end{proof}

\proof[Proof of Theorem \ref{th11}]
Let the inner part of  $f_1$ be the inner part of  the given function $f\in H^p$. We define the outer part of $f_1$  by its modulus on the boundary
 $$
|f_1|(t)=\begin{cases} |f(t)|,& t\in \TT\setminus E\\
\lambda,&t\in E
\end{cases}
$$
as before $E=\{t\in\TT: |f|\ge \lambda\}$. Then
$$
f-f_1=f(1-s),
$$
where $s$ belongs to the Schur class, moreover
\begin{equation}\label{eq31}
s(0)=e^{-\int_E\ln \frac{|f(t)|}\lambda dm}.
\end{equation}

Therefore
\begin{equation*}
\int |f-f_1| dm\le \|f\|\left(\int |1-s|^q dm\right)^{1/q}.
\end{equation*}
We use \eqref{eq22}
\begin{equation}\label{eq32}
\int |f-f_1| dm\le C_p (1-s(0))^{1/q}.
\end{equation}
Since $1-e^{-u}\le u$ and $\ln u\le u^p/p$, from \eqref{eq31} and \eqref{eq32} we get  \eqref{14-9-8}.
\endproof

\begin{remark}
We point out that in view of reiteration theorems (see  Bergh-L\"ofstr\"om \cite{MR0482275}) 
 in our case, as for many other interpolation problems, it suffices to 
apply the counterpart of the estimates of Lemma \ref{l12} in the trivial case, that is, for the value $ q = 2$.
\end{remark}

\section{Holomorphic Random Variables}
In this section we  prepare  the tools of Stochastic Analysis we use for identifying  the complex interpolation spaces $ ( H ^ 1(\O ) , H^\infty(\O ))_{[\theta]}$ and 
$ ( H ^ 1(\TT ) , H^\infty(\TT ))_{[\theta]}$. 
We base this review of  holomorphic random variables on  Varopoulos \cite{var1},  as well as the books by 
Bass \cite{MR1329542} and Durrett \cite{durr}.
We let $ ( \O , ( \cF_t ) ,\PP ) $
denote the filtered Wiener space, and 
we recall  holomorphic martingales on Wiener's filtered probability space $( \O  , ( \cF_t )) $.
Those are defined by their Ito integral representations.
Let  $ (z_t) $ denote
 complex Brownian motion on $ \O $ with normalized covariance process 
$  \la z_t ,\overline{ z_t }\ra = 2 t $ and 
$\la z_t , z_t  \ra = 0  .$
Following Varopoulos \cite{var1}, an integrable $F : \O \to \CC $ 
is called a holomorphic random variable
if there exists a complex valued adapted process $( X_s ) $ so that its Ito integral assumes the form
\begin{equation}\label{6o14-1}
F = F_0 + \int_0^\infty  X_s dz_s .  
\end{equation}
The subspace of $L ^p ( \O ) $ consisting of holomorphic random variables 
is denoted $H^p ( \O ) .$ For a given $ F \in H^p ( \O ) $   and 
\begin{equation}\label{6o14-2}
F_t  = \EE ( F | \cF_t ), 
\end{equation}
we call $ (  F_t )  $  the holomorphic martingale associated to $F . $ 
Combining \eqref{6o14-1} and  \eqref{6o14-2} 
$$ F_t = F_0 + \int_0^t  X_s dz_s  . $$ 
If  $ f \in  H^1(\TT)$  and  $ \t = \inf \{ t > 0 : | z_t | > 1 \}  $ then    
$ F =   f(z_{\t })  $ 
defines a  holomorphic  random variable.  
Since  $z_\t $ is uniformly distributed over $ \TT  $,   
we have  
 $F \in H^1 ( \Omega) $
 and   $  \EE |F|   = \int |f| dm  $.

Holomorphic random variables  are  stable under stopping 
times. Staring with \cite{pfxm1} this property of holomorphic random variables was repeatedly used in approximation problems for spaces of analytic functions, see e.g.
\cite{jopfxm},  %\cite{MR1346251}, 
\cite{MR3254588},   \cite{MR1264825}. 
Given a $ ( \cF_t ) $  stopping time $ \r : \O \to \RR  ^+ $ and its 
generated stopping time $ \s $ algebra $ \cF _\r $,    then 
\begin{equation}\label{6o14-3}
\EE ( F | \cF_\r )  =   F_0 + \int_0^\r  X_s dz_s  =  F_0 + \int_0^\infty 1_{\{s<\r\}} X_s dz_s.\end{equation}
As  $ \r : \O \to \RR  ^+ $  is a stopping time 
the process  $(1_{\{s<\r\}} X_s )$ is adapted to the filtration $ ( \cF_t ) ,$
which verifies the above claim that a stopped holomorphic random variable is again holomorphic. 

Next, holomorphy is  preserved under pointwise multiplication.
If $F , G \in  L^2 ( \O ) ,$  are holomorphic random variables 
with Ito integrals 
$$ F = F_0 + \int_0^\infty  X_s dz_s , \quad 
\text{and} \quad G = G_0 + \int_0^\infty  Y_s dz_s $$
then the covariance formula yields 
 \begin{equation}\label{6o14-4} FG =  F_0 G_0 +  \int_0^\infty ( F_s  Y_s +G_s X_s )dz_s .\end{equation}
Hence $FG $ is a holomorphic 
random variable, and the product  $F_tG_t $ is a holomorphic martingale, 
$$ F_t G_t =\EE ( F G | \cF_t )   . $$
Finally we remark that holomorphy is  preserved under composition with entire functions. 
If $ f : \CC \to \CC $ is analytic  and $F$ is a holomorphic random variable  such that the composition 
$f(F) $ 
is
integrable. Ito's formula gives 
 \begin{equation}\label{6o14-5} 
f(F) = f(F_0) +   \int_0^\infty  \pa f( F_s)  dz_s  .  \end{equation}
Hence  $f(F) $ is a holomorphic 
random variable and   $f( F_t)  $  a holomorphic martingale satisfying
$$  f( F_t) =  \EE ( f(F) | \cF_t ) . $$  
Summing up,   holomorphic random variables are stable under the following operations
\begin{itemize}
\item[-] 
Stopping times, 
\item[-] Pointwise multiplication,
\item[-] Composition with entire functions.
\end{itemize}

\subsection{The stochastic Hilbert transform and outer functions.}
Let $ R = ( R_t ) $ be a  real valued, square integrable martingale on Wiener space with  stochastic integral representation 
$
R_t  = R_0 +  \int_o ^t Y_s dz_s + \overline{Y _s} d  \overline{z}_s  $. 
We define the stochastic Hilbert transform of $R$ 
by putting 
 \begin{equation}\label{1212122} 
 \cH R = i \int_0^ \infty  \overline{Y _s} d  \overline{z}_s -  Y_s dz_s,
 \end{equation} 
 Note that $ \cH R $ is again real valued and  $\cH ^2 R     =   R  - \EE R $.
For   martingales of the form $R = u ( z_\t)$,  where $ u  \in L^2 (\TT ) $ is real valued, 
Ito's formula   connects  the martingale operator $\cH $  to the classical Hilbert transform $ H$ by  the identity      
$ \cH R = ( Hu ) ( z_\t)$,  see  \cite{var1}.

 We also refer to \cite{var1} for the following 

\begin{theorem}\label{1212120}  Let  $ R \in L^2 ( \O ) $ be real valued  with  $ \EE R  = 0 .$  Then 
 $
\EE( R^2 )  = \EE((\cH R) ^2  ) $
and 
$R +  i \cH R \in H^2 (\Omega) . $
\end{theorem} 

In view of the Burkholder-Gundy inequalities,
the stochastic Hilbert transform extends to a bounded operator on $ L ^ p ( \O ) $ for $ 1 < p < \infty  ,$ 
and $ \| \cH R \|_p \le C_p \|R \| _ p $ where $ C_p = C p ^ 2 / ( p-1 ) . $ 
Hence the orthogonal projection  $$ \frac12 (\Id + i \cH) :   L ^ 2 ( \O ) \to H ^ 2 ( \O ) $$ extends boundedly to a projection  on   $L ^ p ( \O )$.  

\begin{theorem}\label{th32}
Let $ R \in L^\infty ( \Omega ) $ be real-valued and   $ F = \exp ( R + i \cH (R) ) . $
Then  $ F  \in H^\infty ( \Omega ) , $ with $|F| = \exp (R )  $ and  $\EE F = \exp\EE ( R )$.
\end{theorem}

For holomorphic martingales Lemma \ref{l12} reads as follows.

\begin{theorem}\label{2-12-15-1}  
Let $ F \in H ^ 2 ( \O ) $ be of the form   $ F =  \exp ( R + i \cH R )$, where $ R $ is real valued. 
Let   $ A = \{ | F | > \lambda \} $ where  $ \lambda > 0 $.  If    $ Z = R 1_{\O \sm A } + (\ln \lambda ) 1_A $ and  
   $ G =  \exp ( Z + i \cH Z ) $,   then $ G\in    H ^ \infty  ( \O ) $  with $ | G | \le \lambda $ and 
\begin{equation}\label{2-12-15-6}\EE | F -G | \le   \EE | F |  ^ 2/  \lambda  .
 \end{equation} 
Moreover if $S = G/F $ then 
\begin{equation}\label{3-9-2}
 \EE | 1 - S | ^ 2 = 2 ( 1 - \EE (S) ) \le 2\EE ( 1_A \ln (| F |/  \lambda)  \le ( 2/ \lambda) \EE ( 1_A | F | ) .\end{equation} 
\end{theorem}
\proof By homogeneity it suffices to consider the case $  \EE | F |  ^ 2 = 1  . $ 
By construction,   $ G\in    H ^ \infty  ( \O ) $   and $ | G | \le \lambda $.
Moreover $ | G|  \le | F | $.  Hence $ ( G/F ) $ is contained in the unit ball of  $  H ^ \infty  ( \O ) $, 
such that  by 
Theorem \ref{th32}
   \begin{equation}\label{2-12-15-7} 
\EE ( G/F )   = \exp \EE \ln (  | G /  F | ) .  \end{equation}  
Since  $ \EE ( G/F ) \in \RR $
we find by arithmetic 
\begin{equation}\label{2-12-15-9} 
\EE | 1 - G/F| ^ 2 = 2 ( 1 -   \EE ( G/F ) ) - \EE ( 1 -| G / F |  ^ 2 ) .  \end{equation}   
%Next we evaluate the expectation of $\EE ( G/F )  . $ 
 Using \eqref{2-12-15-7} and  unwinding the construction of  $G$ we have  
%% $$ \exp \EE \ln (  | G| / | F | ) 
$$ 
\EE ( G/F ) =  \frac{ \EE ( G )}{ \EE ( F )}=  \exp \EE ( Z - R ).
$$
Since $( Z - R ) = 0 $ on $ \O \sm A $ and 
$ ( Z - R ) = \ln  ( \lambda) -   \ln (| F |)  $ on $A$ we have,  the crucial identity,
$$  
\EE ( G/F ) =   \exp -\EE ( \ln (| F | / \lambda  )  1 _ A ) . 
$$ 
The elementary inequalities  $ 1 - \exp ( -t ) \le t $,  $ \ln a < a   $  and  $ \ln a < a ^ 2 / 2  $ for $a \ge 1 $,  together with   $ \EE | F |  ^ 2 = 1 $,  give, 
  \begin{equation}\label{3-9-1} 2(1 -   \EE ( G/F ) )\le  2 \EE ( \ln (|F | /  \lambda  )  1 _ A )  \le \begin{cases}    \EE (|F |^2    1 _ A ) /  \lambda ^ 2  ;\\
                                                                                      2  \EE (|F |    1 _ A ) /  \lambda  
                                                                            \end{cases}          
 \end{equation} 
This proves \eqref{3-9-2}.
Finally, we apply the  Cauchy-Schwarz inequality to the product $ F - G = F ( 1 -  G/F ) $. 
Since $ | G /  F |  \le 1$ we may use  \eqref{2-12-15-9} and  \eqref{3-9-1}
to obtain
$$  \EE | F - G | \le  ( \EE | F |  ^ 2) ^ {1/2 } ( 2 ( 1 -   \EE  G/F  ) )^ {1/2 }  \le    \EE | F |  ^ 2 / \lambda.   $$
  \endproof

%%%%%%%
\subsection{Stopping Times}
\paragraph{The stopping time decomposition.}
Let $ F \in L^2 ( \O ) $ , $M > 1 .$
Put $ \tau _ 0 = 0$ and 
$ \tau_{i+1} = \inf \{ t > \t_i : | \EE (F | \cF _t )| > M^{i+1}\}$.
Then define
\begin{equation}\label{2-9-5}  F_i =  \EE (F | \cF _{\t_i} ) , \quad d_ i = F_{i+1} -F_i . \end{equation}  
We have  $|d_i | \le 2  M^{( 1+i) }$ with  $ {\rm{supp}}\,   d_i \sbe \{\t_i < \infty\} $.
Moreover, $ \{d_i\} $ is a martingale difference sequence, hence   the decomposition $ F = \EE F + \sum_{i = 1 }^\infty d_i   $
converges unconditionally in $L^2  ( \O )  $ satisfying
\begin{equation}\label{2-9-6} \sum_{i = 1 }^\infty  \EE | d_i |^2   = \EE | F - \EE F |^2 . \end{equation}  
If  $ F \in H^2 ( \O ) $ then $ F_i ,  d_ i \in  H^\infty  ( \O ) .  $
We refer to   
$ \{d_i\} $ as the stopping-time decomposition  of $F$
\paragraph{Doob's maximal function.} 
Let $ F \in L^2 ( \O ). $
The maximal function
$ A(F) = \sup_t | \EE (F | \cF _t )| $
satisfies 
$  \|A(F)\|_2  \le C  \|F\|_2 . $
It is linked to  the stopping time decomposition by 
$ | F_i |  \le   A(F),$ and $   |d_i| \le 2  A(F). $
\paragraph{ Outer functions and truncation.}
 Let $ F \in L^2 ( \O )$ and $ \log A(F) \in L^1 ( \O ) . $
Put 
$$ 
R_i = \begin{cases}
           &  M^i \quad {\rm{on}} \quad \{ A(F) > M^i \} ;\\
            &  AF   \quad {\rm{on}} \quad \{ A(F)  \le  M^i \} . 
       \end{cases}
$$
Put  $ \Psi _i = \exp ( \ln R_i + i H\ln R_i ) $,  
$  \Psi  = \exp ( \ln AF + i H\ln AF )$ and 
define 
\begin{equation}\label{2-9-7}
 w_i = \Psi _{i + 8 }  /  \Psi  , \quad E_i = \{ AF > M^i \} . \end{equation}
We have then 
$ w_i \in H^\infty ( \O ) $ with  $  | w_i | \le 1 ,$  and  $  | w_i | \le   | w_{i+1} | .$ 
By Theorem \ref{2-12-15-1}  we get  $L ^ 2 $ inequalities as follows. 
\begin{equation}\label{2-9-8} \EE |1 - w_i |^2 \le 2  (1 - \EE w_i ) \le  \EE( 1_{ E_{i + 8 }} \ln  (A(F)/ M^{i+8}) )   \le 2 \EE( 1_{ E_{i + 8 }} A(F) )M^{-i-8}
\end{equation}
We refer to $\{w_i \} $ as the truncation family associated to $A(F)$.

\subsection{Basic Estimates}
Let $ F \in H^ 1  ( \O ) . $ The stopping-time decomposition 
$ \{d_i\} $ of $F$ and $\{w_i \} $  the truncation family associated to $
A(F)  $ satisfy the following three basic estimates.
\begin{lemma} \label{2-9-3} For any $ F \in H ^ 1 ( \O ) $, its martingale maximal function $A(F)$,  its truncation family $\{w_i \} $ and its  
stopping time decomposition
are related through the pointwise estimates as follows
$$
\sum_{i= 0}^\infty   1_{ E_{i } }|\Psi_i | \le 2 \cdot M \cdot  A(F) .  
$$
\end{lemma}
\proof
Note $ (AF)^{-1} \le M^{-i}$ on $E_i ,  $ and $|\Psi_i| \le M^{i}$ on $E_i , $
hence 
$$\frac{ 1_{ E_{i } }|\Psi_i |}{ A(F)} \le \sum_{j \ge i}   1_{ E_{j } \sm E_{j+1}}
\frac  {M^{i}}{ M^{j} } .$$
Summing   the above estimates over $ i \in \NN $ and evaluating the geometric series $(\sum_{i \le j} M^{i} )$ gives  
$$
\begin{aligned}
\sum_{i= 0}^\infty 1_{ E_{i } }  \frac  {  R_{i}}{ A ( F )  } &   \le \sum_{i= 0}^\infty \sum_{j \ge i}   1_{ E_{j } \sm E_{j+1}}
\frac  {M^{i}}{ M^{j} }  \\
& \le \sum_{j =1}^\infty  M^{-j} 1_{ E_{j } \sm E_{j+1}}(\sum_{i \le j}
M^{i} )  \\
&\le 2M \sum_{j =1}^\infty   1_{ E_{j } \sm E_{j+1}} \le 2M  .
\end{aligned}
$$ 
\endproof

\begin{lemma}  \label{1-9-1}  Let  $ F \in H ^ 1 ( \O ) $.  Then 
$$
\sum_{i= 0}^\infty  M^i  \EE( 1_{ E_{i + 8 }} A(F) ) \le M^{-7} \EE( (
  A(F ) )^2)  .
$$
\end{lemma}
\proof
First 
$
  \EE( 1_{ E_{i + 8 }} A(F) ) \le \sum_{j \ge i+8 }   \PP ( E_{j })M^j
$
holds by elementary  integral estimates. Multiplying by $ M ^ i $ and 
summing over $ i \in \NN $ gives,  
$$
\begin{aligned}
\sum_{i= 0}^\infty  M^i  \EE( 1_{ E_{i + 8 }} AF ) & \le \sum_{i= 0}^\infty \sum_{j \ge i+8 }   \PP ( E_{j })M^{j+ i } \\
& \le \sum_{j \ge 8 }  \PP ( E_{j })M^j ( \sum_{i \le j-8 }  M^i ) \\
& \le M^{-7}  \sum_{j \ge 8 } \PP ( E_{j }) M^{2j} \\
&  \le  M^{-6}  \EE((  A(F) )^2) . 
\end{aligned}
$$
\endproof

\begin{lemma}  \label{2-9-2}   For any $ F \in H ^ 1 ( \O ) $, its martingale maximal function $A(F)$,  its truncation family $\{w_i \} $ and its  
stopping time decomposition  $\{d_i \} $
are related through $L ^ 2 $  estimates by  
$$
 \EE(|\sum_{i= 0}^\infty  d_i (1-w_i) |^2)  \le M^{-6} \EE( (
  AF )^2)  .
$$
\end{lemma}
\proof
Put first  
$$ S_1 =  2 \sum_{i= 0}   \sum_{ j  =1}^{i+7} \EE d_i ( 1-w_i) d_j ( 1-w_j ) \quad\text{and}\quad  S_2 = 2 \sum_{i= 0}   \sum_{ j = i + 8  }^\infty \EE d_i ( 1-w_i) d_j (1-w_j ) $$ 
By direct expansion  gives  $ S_1 + S_2 = \EE(|\sum_{i= 0}^\infty  d_i (1-w_i) |^2) . $ 
We next estimate $ S_1 $ and $ S_2 $  separately.  

Beginning with  $S_1 $, recall that $ | d_i | \le 2 M ^ {i+1} $, and that $\EE | ( 1-w_i)|^2  \le 2  M^{-i-8} \EE( 1_{ E_{i + 8 }} AF ) .$ 
Hence       $$\EE |d_i ( 1-w_i)|^2 \le 8    M^{i-7}\EE( 1_{ E_{i + 8 }} AF )  , $$  
and in view of the arithmetic geometric mean 
$$
\EE |d_i ( 1-w_i) d_j ( 1-w_j
)| \le   2  M^{i-7} \EE( 1_{ E_{i + 8 }} AF ) + 2 M^{j-7} \EE( 1_{ E_{j +8 }} AF ) .$$ 
Taking the sum over $ i \le j \le i + 7 , $ and invoking Lemma \ref{1-9-1} gives 
$$
S_1 \le 2\sum_i \sum_{ j  =1}^{i+7} \EE |d_i ( 1-w_i) d_j ( 1-w_j
)|  \le 7 M^{-7} \EE (( AF )^2). $$

Next we turn to estimating  $S_2 .$ To this end 
fix $ j \ge i +8  $. Note that $d_j $ is supported in $E_j $,  that  $d_i
$ is supported in $E_i $ and that $ E_j \sbe E_i . $  Hence using $ | d_i | \le 2 M ^ i $ and 
$  | 1-w_i| \le 2 $ gives 
$$ \EE |d_i ( 1-w_i) d_j ( 1-w_j
)| \le 16 M^{i+1} M^{j+1} \PP ( E_j ) . $$
Taking the sum  over $  j \ge i + 8$  and invoking Lemma \ref{1-9-1} we get 
the following upper bounds for $ S_2 , $ 
$$
\begin{aligned}
S_2 & \le  \sum_i \sum _{j \,:\,  j \ge  i + 8 } M^{i+1} M^{j+1} \PP ( E_j ) \\
&\le  \sum_{j > 8 }  (\sum _{ i\, : \, j \ge  i + 8 } M^{i+1}) M^{j+1} \PP ( E_j ) \\
% & \le  M^{-6} \sum_{j \ge  8 }   M^{2j} \bP ( E_j ) \\
& \le  M^{-6} \EE (( AF )^2) .
\end{aligned}
$$
\endproof

\section{The Complex Interpolation Space $(H^1(\O) , H^\infty  (\O) ) _{1/2} $}
In this section we give  a  somewhat simplified proof of the result in  \cite{MR1264825} that identifies the complex interpolation space  $(H^1(\O) , H^\infty  (\O) ) _{1/2}$.  
\begin{theorem} \label{2-9-4}
\begin{equation}\label{2-9-10}
   (H^1(\O) , H^\infty  (\O) ) _{1/2}  = H^2  (\O) . \end{equation}
\end{theorem}
\proof 
Fix  $ F \in H^2(\O). $
Let $ \{d_i\} $  be the stopping-time decomposition of $F$  defined by \eqref{2-9-5} and let  
$\{w_i \} $  the truncation family associated to $A(F)$ defined in \eqref{2-9-7}.
Define the following analytic function on  the vertical strip $\overline{S} = \{ \zeta \in \CC : 0 \le \Re \zeta \le 1 \}  $ with values in  $ H^1(\O) $
\begin{equation}\label{2-9-12} G(\o , \zeta ) = \EE F + \sum_{i= 0}^\infty  d_i w_i M^{(1 - 2
  \zeta)( 1+j) } .  \end{equation}
Without loss of generality we assume $ \| F \| _{ H^2  (\O)} = 1 . $ As we discussed in the introduction  (see also   \cite{MR0482275}, or  \cite{MR0167830}),  in order to prove the identity \eqref{2-9-10} it suffices  the following three estimates
\begin{enumerate}
\item $\| G( \cdot , 1/2 ) -  F  \| _{ H^2  (\O)}\le 1/2 ,$
\item $ \sup_{t \in \RR}\| G( \cdot , it ) \| _{ H^1  (\O)} \le C, $
\item $  \sup_{t \in \RR}\| G( \cdot ,1+ it ) \| _{ H^\infty  (\O)} \le C , $
\end{enumerate}
where $ C < \infty $ is a constant  independent of $ F  \in H^2(\O). $
Accordingly the following three lemmata yield Theorem~\ref{2-9-4}.
We exploit Lemma \ref{2-9-2} in the proof of Lemma \ref{2-9-22}, and  Lemma \ref{2-9-3} in the proof of  Lemma \ref{2-9-24}. 
By contract in the proof of Lemma \ref{2-9-14} we exploit just elementary properties of the stopping-time decomposition $ \{d_i\} $ and the  truncation family $\{w_i \}$.

\begin{lemma} \label{2-9-22}
$$
\| G( \cdot , 1/2 ) -  F( \cdot ) \| _{ H^2  (\O)}\le 1/2  \|
F( \cdot ) \| _{ H^2  (\O)} $$
\end{lemma}
\proof
Since  $ F =  \EE F + \sum_{i= 0}^\infty  d_i   $ and since  by \eqref{2-9-12},
$ G(\o , 1/2  ) = \EE F + \sum_{i= 0}^\infty  d_i ( \o) w_i( \o )  ,  $
we have  
$$
\| G( \cdot , 1/2 ) - F \|_2 ^2 = \|\sum_{i= 0}^\infty  d_i (1-
w_i)  \|_2 ^2 . $$ 
By Lemma \ref{2-9-2} we get 
$
\|\sum_{i= 0}^\infty  d_i (1-
w_i)  \|_2 ^2 \le M^{-6} \|A F \| _2^2 $.
Recall  that Doob's maximal theorem  asserts that    $\|A F \| _2\le 2 \| F \| _2 $. For  $ M > 1 $ large enough this finishes the proof.
 \endproof

\begin{lemma} \label{2-9-14}  For each $ t \in \RR , $
$$
\| G( \cdot , it ) \| _{ H^1  (\O)} \le C  \|
F( \cdot ) \| _{ H^2  (\O)} ^2 ,$$
where $C > 0 $ is an absolute constant.
\end{lemma}
\proof By \eqref{2-9-12} we have
 $ G(\cdot , it ) = \EE F + \sum_{j= 0}^\infty  d_j w_j M^{(it(j+1))}M^{( 1+j) } . $
Hence triangle inequality gives 
$$ 
\| G( \cdot , it ) \| _{ H^1  (\O)} \le 
 \EE |F| + \sum_{j= 0}^\infty \| d_j \|_1 M^{( 1+j) } . $$
Since $ \| d_j \|_1 \le 2  \PP ( E_j) M^{( 1+j) } $ and since 
by elementary distributional estimates we have 
$$  \sum_{j= 0}^\infty \PP ( E_j)  M^{( 2j ) }  \le  C_0  M ^ 2 \EE ( (AF)^2), $$
the estimate of  Lemma~\ref{2-9-14} holds with $C = C( M )$. Here we used again that   by Doob's maximal inequality  
$\EE ( (AF)^2) \le 4 \EE (|F|^2)$.
\endproof
\begin{lemma} \label{2-9-24} For each $ t \in \RR , $
$$
\| G( \cdot ,1+ it ) \| _{ H^\infty  (\O)} \le C , $$
where $C > 0 $ is an absolute constant.
\end{lemma}
\proof
By \eqref{2-9-12} we have
$ G(\o , 1 +  it ) = \EE F + \sum_{j= 0}^\infty  d_j w_j M^{(2it(j+1))}M^{-( 1+j) } . $
Hence triangle inequality gives 
 \begin{equation}\label{2-9-18}|G(\o , 1 +  it )| \le 
 \EE |F| + \sum_{j= 0}^\infty  |d_j w_j| M^{-( 1+j) } . \end{equation}
Since $|d_j | \le 2  M^{( 1+j) }$ with  $ {\rm{supp}}\,   d_j \sbe E_j $,  and since $|w_j| \le  |\Psi_j | /A(F) $ we have 
the right hand side of \eqref{2-9-18} bounded by 
$$   \EE |F| + \sum_{j= 0}^\infty   1_{ E_{j } }\frac{|\Psi_j
  |}{A(F)}  . $$
It remains to invoke the pointwise estimates of Lemma \ref{2-9-3} to conclude that $ |G(\o , 1 +  it )| \le C $ with
$ C = C( M ) . $
\endproof 

\bigskip

\noindent
\textbf{\large Acknowledgements}
\bigskip

\noindent
P.F.X.M. and P.Y. were supported by the Austrian Science foundation (FWF) Pr. Nr.'s P28352 and P25591-N25 respectively. P.Y. thanks organizers of the Seventh Jaen Conference on Approximation Theory for their kind invitation and generous hospitality.

%%%%%%%%%%%%%%%%%%%%%%%%%%%%%%%%%%%%%%%%%%%%%%%%%%%%%%%%%%%%%%%%%%%%%%%%%%%%%%%%%%%%%%
%%%%%%%%%%%%%%%%%%%%%%%%%%%%%%%%%%%%%%%%%%%%%%%%%%%%%%%%%%%%%%%%%%%%%%%%%%%%%%%%%%%%%%
%       References
%%%%%%%%%%%%%%%%%%%%%%%%%%%%%%%%%%%%%%%%%%%%%%%%%%%%%%%%%%%%%%%%%%%%%%%%%%%%%%%%%%%%%%
%%%%%%%%%%%%%%%%%%%%%%%%%%%%%%%%%%%%%%%%%%%%%%%%%%%%%%%%%%%%%%%%%%%%%%%%%%%%%%%%%%%%%%

%\begin{thebibliography}

\bibliographystyle{amsplain}

\bigskip

\noindent
Paul F.X. M\"uller,\\
Department of Mathematics,\\
J. Kepler Universit\"at Linz\\
A-4040 Linz\\
{paul.mueller@jku.at}\\ \\
Peter Yuditskii,\\
Abteilung f\"ur Dynamische Systeme\\
und Approximationstheorie,\\
J. Kepler Universit\"at Linz, A-4040 Linz\\
petro.yudytskiy@jku.at%\\ \\
%
%}

\end{document}